\documentclass[12pt,a4paper]{article}

\usepackage{rawfonts}
\input prepictex
\input pictexwd
\input postpictex

\newif\ifpdf
\ifx\pdfoutput\undefined
\pdffalse 
\else
\pdfoutput=1 
\pdftrue
\fi

\ifpdf
\usepackage[pdftex]{graphicx}
\else
\usepackage{graphicx}
\fi

\usepackage{amssymb}
\usepackage{rotating}
\usepackage{amsmath}
\usepackage[pdftex]{graphicx}
\usepackage{epsfig}
\usepackage[pdftex]{color}

\newcommand{\fhat}[1]{\widehat{#1}}
\newcommand{\Z}{\mathbb{Z}}

\newcommand{\spa}{{\rm span \,}}
\newcommand{\supp}{{\rm supp}\, }

\newcommand{\R}{\mathbb{R}}
\newcommand{\Q}{\mathbb{Q}}

\newcommand{\N}{\mathbb{N}}

\newcommand{\C}{\mathbb{C}}

\newcommand{\calS}{{\cal S}}

\newcommand{\vol}{ {\rm vol }\, }

\newcommand{\diag}{\text{diag}}

\newcommand{\LtR}{L^2(\mathbb{R}^d)}

\newtheorem{theorem}{Theorem}[section]
\newtheorem{lemma}[theorem]{Lemma}
\newtheorem{proposition}[theorem]{Proposition}
\newtheorem{definition}[theorem]{Definition}
\newtheorem{remark}[theorem]{Remark}
\newtheorem{example}[theorem]{Example}

\newtheorem{corollary}[theorem]{Corollary}

\title{A geometric construction of tight Gabor frames with multivariate compactly supported smooth windows}
\author{G\"otz E. Pfander, Peter Rashkov, Yang Wang}

\begin{document}

\maketitle

\begin{abstract}
The geometry of fundamental domains of lattices was used by Han and
Wang to construct multivariate Gabor frames for separable lattices. We
build upon their results to obtain Gabor frames with smooth and compactly
supported window functions. For this purpose we study pairs of lattices
which have equal density and allow for a common compact and star-shaped
fundamental domain. The results are then extended to a larger class of
lattices via symplectic equivalence.
 \footnote{ {\it Math Subject Classifications:}\  42C15, 42C40 }
\end{abstract}

\section{Introduction}

Gabor systems are useful to construct discrete time-frequency
representations of signals. A Gabor system is customarily denoted as
$(g,\Lambda)=\{M_\omega T_x g:(x,\omega)\in\Lambda\}$, where the
square integrable function $g$ is the Gabor window, $\Lambda=M\Z^{2d}$,
$M$ full rank, a lattice  in the time-frequency space $\R^{2d}$, $T_x$ the
translation operator $(T_x f)(t) = f(t-x),\,x\in \R^d$, and $M_\omega$ the
modulation operator $(M_\omega f)(t)= e^{2\pi
i\langle\omega,y\rangle}f(t),\,\omega\in \R^d$. 
A Gabor system is a tight frame for the space of square integrable functions
$L^2(\R^d)$ if up to a scalar factor, Parseval's identity holds. That is, for
some $c>0$, we have
\begin{equation}\label{eq:eqngab}
f=c \, \sum_{(x,\omega)\in\Lambda}\langle f, M_\omega T_x g   \rangle M_\omega T_x g , \quad f\in L^2(\R^d),
                                           \end{equation}
where the series converges unconditionally and  the computation of the representation coefficients is stable.


Necessary and sufficient conditions for the Gabor system $(g,\Lambda)$ to
be a frame for $L^2(\R^d)$ have a long history in applied harmonic analysis.
A classical statement is the density condition, which states that any Gabor
system whose lattice $\Lambda=M\Z^{2d}$ has density $d(\Lambda)=1
/|\det M| <1$  is not complete and therefore not a Gabor frame for
$L^2(\R^d)$ \cite{Dau90,Dau92,Jan94}. Recently Bekka showed that the
density condition is sufficient as well: for each $\Lambda$ with
$d(\Lambda)\ge1$ there exists a function $g\in L^2(\R^d)$ such that
$(g,\Lambda)$ is a frame for $L^2(\R^d)$ \cite{Bek04}. In his fundamental
paper, Bekka proves existence only and his work reveals nothing about the
window besides membership in $L^2(\R^d)$. Multivariate Gabor frames are
constructed more explicitly for separable lattices in \cite{HW01,HW04},
however, the obtained Gabor windows are characteristic functions on sets
that are fundamental domains for pairs of lattices. These fundamental
domains may well be unbounded; in this case the constructed Gabor
windows decay neither in time nor in frequency, so \eqref{eq:eqngab} is not
local. For any lattice $\Lambda$, the existence of so-called multi-window
Gabor frames with windows in the Schwarz space or in modulation spaces
has been shown in \cite{Rie88,Lue05}, but the number of windows needed
does not follow from their analysis, for example, it is not clear whether
$d(\Lambda)>1$ implies that a single window in the Schwarz space suffices.

In this paper, we further explore the results of Han and Wang and use our
findings to construct tight Gabor frames $(g,\Lambda)$ in $L^2(\R^2)$ for
separable lattices $\Lambda=\Lambda_1\times\Lambda_2\subset\R^4$ with
nonnegative, smooth and compactly supported windows, that is, $g\in
C_c^\infty(\R^2)$. To achieve this, we shall assume star-shapedness of the
common fundamental domain of a pair of lattices. Note that the so-called
Balian-Low theorem \cite{BHW98, GHHK03} implies that if $(g,\Lambda)$ is
a frame for $L^2(\R^d)$ with $g$ being in the Schwarz space $\calS(\R^d)$,
then $d(\Lambda)>1$. Hence,
 our analysis is limited to lattices with density greater than 1.


The characterization of lattices $\Lambda$ so that $(g,\Lambda)$ is a frame
for $L^2(\R^d)$ for a fixed $g$ are rare in the literature
\cite{Lyu92,SW92,JS02,Jans}, and with the exception of \cite{GroGa, PRtr}
these concern the case $d=1$ only. In principle, characterizing  spanning
properties of Gabor frames for lattices $\Lambda$ of higher dimension
($\Lambda\subset\R^{2d},d\ge 2$) is much more intricate than in the
one-dimensional case.

The paper is organized as follows. Section~\ref{section:theory} recalls basic
facts from Gabor analysis and from the Fourier theory of translational tilings
that are used in this paper.  Section~\ref{section:charwind} contains results
from \cite{HW01,HW04} as well as minor extensions of these. The
construction of smooth and compactly supported Gabor windows for a class
of separable lattices is presented in Section~\ref{section:smoothwind}.
These results extend to some lower block-triangular matrices,
Section~\ref{section:examples} contains examples in the bivariate case, that
is, we present pairs of lattices in $\R^4$ which allow for a common
star-shaped fundamental domain.

\section{Tools in time-frequency analysis}\label{section:theory}

Throughout the text, all subsets of $\R^d$ are assumed to be Lebesgue
measurable. For $\Omega \subseteq \R^d$, we use $\Omega +x=\{y+x:\ y\in
\Omega\}$, $x\in \R^d$,  $\gamma \Omega =\{\gamma y:\ y\in \Omega\}$,
$\gamma
>0$, and $ \Omega_\epsilon =\{y+z:\ y\in \Omega, \ \|z\|_2<\epsilon\}$, $\epsilon
>0$, where $\|\ \|_2$ denotes the Euclidean norm on $\R^d$. The characteristic function on
$\Omega \subseteq \R^d$ is denoted by $\chi_\Omega$, that is, we have
$\chi_\Omega(x)=1$ if $x\in\Omega$ and $\chi_\Omega(x)=0$ else.

We use the standard notation for complex-valued functions on $\R^d$:
$\calS(\R^d)$ is the Schwarz space, and $C_c^\infty(\R^d)$ the space of
smooth and compactly supported  functions.

A \emph{time-frequency shift} is $ (\pi(\lambda)f)(t)=(M_\omega T_x
f)(t)=f(t-x)e^{2\pi i\langle \omega,t\rangle}$ for
$\lambda=(x,\omega)\in\R^{2d}$. The Fourier transformation used here is
$L^2(\R^d)$--normalized, that is, $\fhat f(\xi)=\int_{\R^d} f(y)e^{-2\pi i\xi\cdot
y}dy$ for $f$ integrable.


\begin{definition}
  \label{def:lattice} A  {lattice} $\Lambda$ in $\R^{2d}$ is a discrete subgroup of the additive group
$\R^{2d}$, that is $\Lambda=M\Z^{d}$, with $M$ being full-rank ($\det M\neq
0$). A separable lattice has the form $\Lambda=A\Z^d{\times} B\Z^d$.
The  {dual lattice} of $\Lambda=M\Z^{d}$ is $\Lambda^\bot=M^{-T}\Z^{2d}$,
and the
 {adjoint lattice} of $\Lambda=M\Z^{2d}$ is $\Lambda^\circ=\{\lambda\in
\R^d\times\fhat\R^d: \pi(\lambda)\pi(\mu)=\pi(\mu)\pi(\lambda)\text{ for all
}\mu\in\Lambda\}$. The  {volume} of the lattice $\Lambda=M\Z^d$ equals the
Lebesgue measure of $\R^{d}/\Lambda$, that is
$\vol\Lambda=m(\R^{d}/\Lambda)=|\det M|$, and the  {density} of
$\Lambda$ is $d(\Lambda)=(\vol\Lambda)^{-1}$.
\end{definition}
 We have
$d(\Lambda^\bot)=d(\Lambda^\circ)=1/ {d(\Lambda)}$, and the adjoint
$\Lambda^\circ$ of a separable lattice $\Lambda=A\Z^d{\times} B\Z^d$  is
$\Lambda^\circ=(B\Z^d)^\bot{\times}(A\Z^d)^\bot=(B^{-T}\Z^d){\times}(A^{-T}\Z^d)$,
where $M^{-T}$ denotes $(M^T)^{-1}$. Moreover, $(M^\circ)^\circ=M$ and
$(M^\perp)^\perp=M$ \cite{G01}.

\begin{definition}\label{def:Fundamental domain}
Let $\Omega\subset\R^d$ be measurable and let $\Lambda$ be a full rank
lattice. If $(\Omega+\lambda_1)\cap(\Omega+\lambda_2)$ is a null set for
$\lambda_1\neq \lambda_2, \lambda_1, \lambda_2\in\Lambda$, then
$\Omega$ is a packing set  for $\Lambda$. If in addition, $\R^d
=\cup_{\lambda\in\Lambda}(\Omega+\lambda)$, then $\Omega$ is a tiling
set (fundamental domain) for $\Lambda$. \end{definition}

Clearly, a measurable set $\Omega$ is a packing set  if and only if
$\sum_{\lambda\in\Lambda}\chi_\Omega(x-\lambda)\le 1$ for almost all
$x\in\R^d$ and a fundamental domain for $\Lambda$ if and only if
$\sum_{\lambda\in\Lambda}\chi_\Omega(x-\lambda)=1$ for almost all
$x\in\R^d$. Furthermore, if $\Omega$ is a packing set  and
$m(\Omega)=\vol\Lambda$, then $\Omega$ is a fundamental domain for
$\Lambda$. Moreover, any translate of a fundamental domain of $\Lambda$
is a fundamental domain of $\Lambda$.

If $\Omega$ is a star-shaped fundamental domain for $\Lambda$, that is,
there exists a point $N\in\Omega$ such that for all $Q\in \Omega$ the line
segment $\overrightarrow{NQ}$ is contained entirely within $\Omega$ (see
Figure~\ref{figure:dilation-omega}), then $\gamma (\Omega-Q)$ is a packing
set for $\Lambda$ for all $0<\gamma \leq 1$.

Fundamental domains for lattices can be
characterized with methods from Fourier analysis \cite{K01}. 
\begin{lemma}\label{lemma:vanishing} Let $\Lambda$ be a lattice in $\R^{d}$ and $\Omega\subset\R^d$ be measurable.
$\Omega$ is a fundamental domain for $\Lambda$ if and only if $\fhat{\chi_\Omega}$ vanishes on
$\Lambda^\bot{\setminus}\{0\}$.
\end{lemma}

Moreover, the following result is central to our analysis
\cite{Fu74,IKT03,KM06}.

\begin{theorem}
\label{th:Fug}  The set $\Omega$ is a fundamental domain of the lattice
$\Lambda$ in $\R^d$ if and only if the set of pure frequencies $\{e^{2\pi i
\lambda x}\}_{\lambda \in \Lambda^\perp}$ is an orthonormal basis for
$L^2(\Omega)$.
\end{theorem}

We recall some important definitions and properties from the theory of
Gabor frames and Gabor Riesz basis  sequences.
\begin{definition}\label{def:frame}
A Gabor system $(g,\Lambda)$ is a frame for $L^2(\R^d)$ with
frame bounds $0<a\le b$ if
\begin{equation}\label{eq:frameineq}
a\|f\|_{L^2}^2\leq\sum_{\lambda\in\Lambda}|\langle
f,\pi(\lambda)g\rangle|^2\leq b\|f\|_{L^2}^2\quad\text{for all }f \in L^2(\R^d).
\end{equation}
The Gabor frame is tight if $a=b$ is possible.
A Gabor system $(g,\Lambda)$ is a Riesz basis for $L^2(\R^d)$
if there exist constants $0<a\le b$ such that
\begin{equation}\label{eq:RBineq}
a\|\mathbf{c}\|_{\ell^2}^2\leq\Big\|\sum_{\lambda\in\Lambda}c_\lambda
\pi(\lambda)g\Big\|_{L^2}^2\leq b\|\mathbf{c}\|_{\ell^2}^2\quad\text{for all }\mathbf{c}\in\ell^2(\Lambda),
\end{equation}
and $\overline{\spa(g,\Lambda)}=L^2(\R^d)$. A Gabor Riesz basis
sequence is a Riesz basis for the $L^2$-closure of its linear span.
\end{definition}
Clearly, if $(g,\Lambda)$ is a Riesz basis sequence with $\|g\|_{L^2}=1$
and $a=b=1$ in \eqref{eq:RBineq} then $(g,\Lambda)$ is is an orthonormal
sequence. Moreover, we shall use below that if $(g,\Lambda)$ is a tight
Gabor frame, then $a=b=d(\Lambda)\|g\|^2_{L^2}$ in  \eqref{eq:frameineq}
and $c=(d(\Lambda)\|g\|^2_{L^2})^{-1}$ in \eqref{eq:eqngab} (Theorem 5 in
\cite{BCHL06b}).

The usefulness of Gabor frames and tight Gabor frames stems from the
following result.
\begin{theorem}
Let $g\in L^2(\R^d)$ and let  $\Lambda$ be a full rank lattice. If
$(g,\Lambda)$ is a frame for $L^2(\R^d)$, then exists a so-called dual
window $\gamma \in L^2(\R^d)$ with
\begin{equation}\label{eq:framereconstruction}
  f=\sum_{\lambda\in\Lambda} \langle f,\pi(\lambda)\gamma \rangle \pi(\lambda) g
  =\sum_{\lambda\in\Lambda} \langle f,\pi(\lambda)g \rangle \pi(\lambda) \gamma,\quad f\in L^2(\R^d)\,.
\end{equation}

If $(g,\Lambda)$ is a tight frame for $L^2(\R^d)$, then we can choose in
\eqref{eq:framereconstruction}
\\ $\gamma=(d(\Lambda)\|g\|^2_{L^2})^{-1}\,g$, that is, \eqref{eq:eqngab}
holds with $c=(d(\Lambda)\|g\|^2_{L^2})^{-1}$.
\end{theorem}

The benefit of having compactly supported and smooth $g$ in
\eqref{eq:eqngab} and \eqref{eq:framereconstruction}   is clear.  Only then,
we can guarantee that for any $r\in \N$ and any compactly supported
$r$-times differentiable function $f$, there exists $N_f,C_f >0$ with  $\langle
f,\pi(x,\omega) g \rangle =0$ whenever $\|x\|_2\geq N_f$, and $|\langle
f,\pi(x,\omega) g \rangle|\leq C |\omega|^{-r}$. This property also allows for
efficient quantization of the expansion coefficients in \eqref{eq:eqngab} and
\eqref{eq:framereconstruction} \cite{Yil03}.

For a detailed discussion of Bessel sequences, Gabor frames and Riesz
basis sequences, we refer to \cite{G01,C03}.
%

The following results are central in the theory of Gabor frames
 \cite{RS97,FG97,FZ98}.
\begin{theorem}\label{th:ronshen}
Let $g\in L^2(\R^d)$ and let $\Lambda$ be a full rank lattice. Then
$(g,\Lambda)$ is a frame for $L^2(\R^d)$ if and only if $(g,\Lambda^\circ)$
is a Riesz sequence. Moreover, $(g,\Lambda)$ is a tight frame for
$L^2(\R^d)$ if and only if $(g,\Lambda^\circ)$ is an orthonormal sequence.
\end{theorem}

Gabor frame theory is rooted in the representation theory of the
Weyl-Heisenberg group, a fact that we shall exploit in
Sections~\ref{section:charwind} and \ref{section:smoothwind}
\cite{Fol89,G01}. In particular, we shall use so-called metaplectic operators
which are discussed below.

\begin{definition}
  The symplectic group ${\rm Sp}(d)$ is the subgroup of $GL(2d,\R)$ whose members $\left(\begin{smallmatrix}
     A & C \\
     D & B \\
   \end{smallmatrix}
 \right)$ are characterized by $AD^T=A^TD$, $BC^T=B^TC$ and $A^TB-D^TC=I$.
 A  lattice $\Lambda=M\Z^{2d}$ with $M\in{\rm Sp}(d)$ is called symplectic lattice.
\end{definition}

\begin{theorem}\label{th:metaplectic}
  For $M\in {\rm Sp}(d)$ exists a unitary operator  $\mu(M)$ on $L^2(\R^d)$, a so-called metaplectic operator, with
  $\pi(M \lambda)=\mu(M)^\ast \pi(\lambda) \mu(M)$, $\lambda\in\R^{2d}$.
\end{theorem}

\begin{theorem}\label{th:symplectic lattices}
Let $\Lambda$ be a full rank lattice in $\R^{2d}$ and $M$ be a symplectic
matrix in $GL(\R,2d)$. Then the following are equivalent:
\begin{enumerate}
\item There exists $g\in L^2(\R^d)$, respectively $g\in\calS(\R^d)$, such
    that $(g,\Lambda)$ is a Gabor frame for $L^2(\R^d)$.
\item There exists $\tilde g\in L^2(\R^d)$, respectively $\tilde
    g\in\calS(\R^d)$, such that $(\tilde g,M\Lambda)$ is a Gabor frame for
    $L^2(\R^d)$.
\end{enumerate}
\end{theorem}
Theorem~\ref{th:symplectic lattices} follows from
Theorem~\ref{th:metaplectic} and the choice   $\tilde g=\mu(M)g$.  In fact,
all metaplectic operators restrict to $\calS(\R^d)$ \cite{G01}, but
unfortunately not to $C_c^\infty(\R^d)$ as discussed in the next paragraph. In
general, the spanning properties of the Gabor system $(g,\Lambda)$ are
transferred to the Gabor system $(\mu(M) g,M\Lambda)$ \cite{G01}. Thus,
we can replace the quantifier `a Gabor frame for $L^2(\R^d)$' by `Riesz
basis sequence' or `frame sequence'.

Note that ${\rm Sp}(d)$ is generated by matrices of the form
$\left(\begin{smallmatrix}
    0 & I \\
    -I & 0 \\
   \end{smallmatrix}
 \right), \left(\begin{smallmatrix}
     B & 0 \\
     0 & B^{-T} \\
   \end{smallmatrix}
 \right)$, $\det B\neq 0$, and $
\left(\begin{smallmatrix}
    I & 0 \\
    C & I \\
   \end{smallmatrix}
 \right)
$, $C$ - symmetric. The corresponding metaplectic operators are,
respecitively, the Fourier transform, the normalized dilation $f\mapsto |\det
B|^{1/d}\, f\circ B$, and the multiplication with a chirp $e^{\pi i \langle
x,Cx\rangle}$. Clearly, $\mu(M)$ restricts to $C_c^\infty(\R^d)$ if $M$ is
generated by matrices of the form $\left(\begin{smallmatrix}
     B & 0 \\
     0 & B^{-T} \\
   \end{smallmatrix}
 \right)$, and $
\left(\begin{smallmatrix}
    I & 0 \\
    C & I \\
   \end{smallmatrix}
 \right)$,  $C$ - symmetric, a fact that we shall exploit below.

\section{Characteristic functions as  Gabor frame windows}\label{section:charwind}

Han and Wang construct Gabor systems with windows that are characteristic
functions on fundamental domains of pairs of lattices \cite{HW01,HW04}.
Their construction is based on  part 1 implies part 3 of the following result.

\begin{proposition}\label{lem:necessity}
  Let $\Omega\subseteq \R^d$ be a fundamental domain for $A \Z^d$. The following are equivalent.
  \begin{enumerate}
\item $\Omega$ is a packing set  for $B^{-T}\Z^d$;
\item $(\chi_\Omega, A\Z^d\times B\Z^d)$ is a frame for $L^2(\R^d)$;
\item $(\chi_\Omega, A\Z^d\times B\Z^d)$ is a tight frame for $L^2(\R^d)$
    with frame bound $1/|\det B|$;
\item  $(\chi_\Omega, B^{-T}\Z^d\times A^{-T}\Z^d)$ is a Riesz basis
    sequence;
\item  $((m(\Omega)^{-1/2}\, \chi_\Omega, B^{-T}\Z^d\times A^{-T}\Z^d)$ is
    an orthonormal sequence.
  \end{enumerate}
\end{proposition}

Note that $\Omega$ being a fundamental domain for $A\Z^d$ and a packing
set for $B^{-T}\Z^d$ implies $|\det B^{-T}| \geq |\det A|$ and for
$\Lambda=A\Z^d\times B\Z^d$, we then have $d(\Lambda)\geq1$ and for
$\Lambda=B^{-T}\Z^d\times A^{-T}\Z^d$, we have $d(\Lambda)\leq1$.
\begin{proof}
We shall  show that  part 1 implies part 5,  and  part 4 implies part 1. Clearly,
part 3 implies part 2 and part 5 implies part 4. The equivalence of parts 2
and 4 and parts 3 and 5 follows from Theorem~\ref{th:ronshen}. The explicit
frame bound in part 3 follows from $\|\chi_\Omega\|^2_{L^2}=|\det A|$ and
$d(A\Z^d\times B\Z^d)=1/|\det A\, \det B|$ \cite{BCHL06b}.

Note that $\Omega$ being a packing set  for $B^{-T}\Z^d$ implies
$$
\langle \pi(x,\omega) \chi_\Omega,
             \pi(x',\omega') \chi_\Omega\rangle =0, \quad(x,\omega),(x',\omega')
             \in B^{-T}\Z^d\times A^{-T}\Z^d,\text{ with } x\neq x'\,.
$$
Moreover, the family $\{M_{\omega}\chi_{\Omega + x}: \ \omega \in A^{-T}
\Z^n\}$ is an orthogonal basis for $L^2(\Omega+x)$, $x\in\R^d$, by
Theorem~\ref{th:Fug}, so part 4 follows.

Now, if $\Omega$ is not a packing set  for $B^{-T}\Z^d$, then there exists
$x\in B^{-T}\Z^d$, $W=\Omega\cap (\Omega+x)$ with $m(W)>0$.  Note that
$W-x= (\Omega-x)\cap \Omega\subseteq \Omega$. As $\Omega$ is a
fundamental domain for $A\Z^d$, there exists $\{c_\omega \}_{\omega \in
A^{-T}\Z^d }$ with
$$
\chi_W + \chi_{W-x}= \sum_{\omega\in A^{-T}\Z^d } c_\omega M_\omega \chi_\Omega\,.
$$
For $N\in\N$, consider
\begin{eqnarray}
  f_N &=& \sum_{k=1}^{N}\sum_{\omega\in A^{-T}\Z^d } (-1)^k c_\omega T_{kx}M_\omega  \chi_{\Omega}\notag \\
   &=& \sum_{k=1}^{N} (-1)^k (\chi_{W+kx}+  \chi_{W+(k-1)x} )= (-1)^N\chi_{W+Nx} -\chi_{W} \notag .
\end{eqnarray}
Clearly, $\|f_N\|^2_{L^2}\leq 2m( W)$, but the coefficients
$\{d_{x,\omega}\}_{(x,\omega)\in B^{-T}\Z^d\times A^{-T} \Z^d}$ of the Gabor
expansion of $f_N$ in terms of $(\chi_\Omega, B^{-T}\Z^d\times
A^{-T}\Z^d)$ satisfy
$$ \sum_{(x,\omega)\in B^{-T}\Z^d\times A^{-T}\Z^d }|d_{x,\omega}|^2
        =\sum_{k=1}^N \sum_{\omega \in A^{-T}\Z^d } |c_\omega|^2 =
        N \sum_{\omega \in A^{-T}\Z^d } |c_\omega|^2 \geq N m(W). $$

This implies that $(\chi_\Omega, B^{-T}\Z^d\times A^{-T}\Z^d)$ is not a Riesz
basis sequence.


\end{proof}

Han and Wang's central result is the following.

\begin{theorem}\label{th:hanwang}
Let $A\Z^d$ and $B\Z^d$ be two full-rank lattices in $\R^d$, such that $|\det
A|=|\det B|$. Then there exists a measurable set $\Omega$ which is a
fundamental domain for both $A\Z^d$ and $B\Z^d$. If $|\det A|\ge|\det B|$,
then there exists a measurable set $\Omega$, which is a  packing set  for
$A\Z^d$ and a tiling set for $B\Z^d$.
\end{theorem}

The combination of Proposition~\ref{lem:necessity} and
Theorem~\ref{th:hanwang} provides us with the main result from
\cite{HW01,HW04}. The results were the first to prove the existence of
Gabor frame windows for any separable lattice with density greater than or
equal 1.

\begin{theorem}\quad\label{th:ONBexist}
\begin{enumerate}
 \item If $|\det A  \det B|=1$, then there  exists $\Omega\subseteq \R^d$
     such that $(\chi_\Omega,A\Z^d\times B\Z^d)$ is an orthonormal basis
     for $L^2(\R^{d})$.
 \item If $|\det A\det B|\ge 1$, then there exists $\Omega\subseteq \R^d$
     such that $(\chi_\Omega,A\Z^d\times B\Z^d)$ is an orthogonal
     sequence for $L^2(\R^d)$.
\item If $|\det A\det B|\le 1$, then there exists $\Omega\subseteq \R^d$
    such that $(\chi_\Omega,A\Z^d\times B\Z^d)$ is a tight Gabor frame for
    $L^2(\R^d)$ with frame bound $1/|\det B|$.
\end{enumerate}
\end{theorem}
\begin{proof}
Parts 2 and 3 follow directly from Theorem~\ref{th:hanwang} and
Lemma~\ref{lem:necessity}. Part 1 follows from part 2 and the observation
that $\Omega$ tiles with respect to $A\Z^d$ and $B^{-T}\Z^d$, then
$(\chi_\Omega,A\Z^d\times B\Z^d)$ is complete.
\end{proof}

Proposition~\ref{lem:necessity} states that if  $\Omega\subseteq \R^d$ is a
fundamental domain for $A \Z^d$, respectively $B^{-T}\Z^d$, then
$\Omega$ must be a packing set  for  $B^{-T} \Z^d$, respectively $A \Z^d$,
in order for $(\chi_\Omega, A\Z^d\times B\Z^d)$ to be a frame, respectively
a Riesz basis  sequence. The interplay of the fundamental domain property
of one lattice  with the packing set  property with respect to a second lattice
is further illuminated by the following observation.

\begin{proposition} Let $\Omega\subseteq \R^d$.

\begin{enumerate}
\item If $(\chi_\Omega, A\Z^d\times B\Z^d)$ is a  frame for $L^2(\R^d)$,
    then  $\Omega$ contains a fundamental domain of $A\Z^d$.

  \item If $\Omega$ is a packing set  for $B^{-T} \Z^d$, then
      $(\chi_\Omega, A\Z^d\times B\Z^d)$ is a  frame for $L^2(\R^d)$ if and
      only if $\Omega$ contains a fundamental domain of $A\Z^d$.
\item If $\Omega$ is a packing set  for $B^{-T} \Z^d$, then $(\chi_\Omega,
    A\Z^d\times B\Z^d)$ is a tight frame for $L^2(\R^d)$ if and only if
    $\Omega$ is the union of $k\in\N$  disjoint fundamental domains of
    $A\Z^d$. The frame bound is then $k/|\det B|$.

\item If  $(\chi_\Omega, A\Z^d\times B\Z^d)$ is a  Riesz basis sequence,
    then  $\Omega$ contains a fundamental domain of $B^{-T}\Z^d$.

\item If $\Omega$ is a packing set  for $A \Z^d$, then $(\chi_\Omega,
    A\Z^d\times B\Z^d)$ is a  Riesz basis sequence if and only if
    $\Omega$ contains a fundamental domain of $B^{-T}\Z^d$.
\item  If $\Omega$ is a packing set  for $A \Z^d$, then $(\chi_\Omega,
    A\Z^d\times B\Z^d)$ is an orthogonal sequence if and only if $\Omega$
    is the union of a finite number of fundamental domains of $B^{-T}\Z^d$.
  \end{enumerate}
\end{proposition}
\begin{proof}
Theorem~\ref{th:ronshen} implies that it suffices to show parts 1, 5, and 6.
The first statement is trivial, as else $(\chi_\Omega, A\Z^d\times B\Z^d)$
would not span $L^2(\R^d)$.

In the remaining parts, we assume that $\Omega\subseteq \R^d$ is a
packing set  for $A \Z^d$, hence, it suffices to show that $\{M_\omega
\chi_{\Omega}\}_{\omega\in B\Z^d}$ is a Riesz basis sequence if and only if
$\Omega$ contains a fundamental domain of $B^{-T}\Z^d$, respectively an
orthogonal sequence if and only if $\Omega$ is the union of a finite number
of fundamental domains of $B^{-T}\Z^d$.

If $\Omega$ contains a fundamental domain of $B^{-T}\Z^d$, then clearly
$\{M_\omega \chi_{\Omega}\}_{\omega\in B\Z^d}$ is a Riesz basis
sequence \cite{C03}. If $\Omega$ does not contain a fundamental domain
of $B^{-T}\Z^d$, then $\sum_{\omega \in B^{-T}\Z^d} \chi_{ \Omega +
\omega}=0$ on a set $W\subseteq  B^{-T}[0,1)^d$ with $m(W)>0$. Note that
we have also $\sum_{\omega \in B^{-T}\Z^d} \chi_{ \Omega + \omega}=0$
on a set $W+x$ for any $x\in B^{-T}\Z^d$.

As $B^{-T}[0,1)^d$ is a fundamental domain for $B^{-T}\Z^d$, there exist
$\{c_\gamma\}_{\gamma \in B\Z^d}\neq 0$ such that $\chi_W=\sum
c_\gamma M_\gamma \chi_{B^{-T}[0,1)^d}$. Then $\sum c_\gamma
M_\gamma \chi_{\R^d}= \sum_{\omega\in B^{-T}\Z^d } \chi_{W+\omega}$
and
$$0=\chi_\Omega\big( \sum_{\omega\in B^{-T}\Z^d } \chi_{W+\omega}\big)=
\sum_{\gamma\in B\Z^d} c_\gamma M_\gamma \chi_{\Omega}. $$
We conclude that
$\{M_\omega \chi_{\Omega}\}_{\omega\in B\Z^d}$ is not a Riesz basis.

Now, if $\Omega$ is the finite union of fundamental domains of
$B^{-T}\Z^d$, then clearly $\{M_\omega \chi_{\Omega}\}_{\omega\in B\Z^d}$
is an orthogonal sequence. On the other side, if  $\{M_\omega
\chi_{\Omega}\}_{\omega\in B\Z^d}$ is an orthogonal sequence, then for
$\omega\neq 0$,
\begin{eqnarray}
  \notag
 0&=&    \int  M_\omega \chi_\Omega(t) \,dt= \int e^{2\pi i \omega t} \chi_\Omega(t)\,dt \\
  &=&\sum_{p\in B^{-T}\Z^d}\int_{B^{-T}[0,1)^d +p}e^{2\pi i \omega t} \chi_\Omega(t)\,dt \notag \\
    &=&\int_{B^{-T}[0,1)^d}\sum_{p\in B^{-T}\Z^d}e^{2\pi i \omega (t-p)} \chi_\Omega(t-p)\,dt \notag \\
       &=&\int_{B^{-T}[0,1)^d} e^{2\pi i \omega t} \sum_{p\in B^{-T}\Z^d}\chi_\Omega(t-p)\,dt \,.
\end{eqnarray}
Hence, $ \sum_{p\in B^{-T}\Z^d}\chi_\Omega(t-p)$ is an integer-valued
constant function, that is, \\  $\sum_{p\in
B^{-T}\Z^d}\chi_\Omega(t-p)=k\in\N$. This implies that almost every point in
$\R^d$ is covered by $k$ translates of $\Omega$. Hence, $\Omega$ is the
union of $k$ fundamental domains.  In fact, we can find a  fundamental
domain $\Omega_1 \subseteq \Omega$, remove it as shown below, and
then continue inductively.

As long as $\sum_{p\in B^{-T}\Z^d}\chi_\Omega(t-p)>1$ on a set of nonzero
measure, there exists $q\in B^{-T}\Z^d$ such that $\Omega+q \cap \Omega
$ has nonzero measure.  Set $\Omega'=\Omega\setminus ( \Omega + q)$.
Now,
$$\bigcup_{p\in B^{-T}\Z^d} \big( \Omega +  p \big)= \bigcup_{p\in B^{-T}\Z^d} \big(  \Omega' +  p \big)$$ follows from
$\bigcup_{\ell\in\Z}  \big( \Omega + \ell q \big)  = \bigcup_{\ell\in\Z} \big(
\Omega' + \ell q \big)$, which in turn follows from $ \Omega \subseteq
\bigcup_{\ell\in\Z} \big(  \Omega' + \ell q \big)$. But if $x\in \Omega\setminus
\Big( \bigcup_{\ell\in\Z} \big(  \Omega' + \ell q \big)\Big)$, then $x\in  \Omega
+ \ell q$ for all $\ell\in \N$, a contradiction.
\end{proof}

\begin{remark}
\rm  For $(\chi_\Omega, A\Z^d\times B\Z^d)$ to be a  frame for $L^2(\R^d)$
it is not necessary that $\Omega\subseteq \R^d$ is a packing set  for
$B^{-T} \Z^d$. In fact, $(\chi_{[0,3/2)}, \ 1/2 \Z\times \Z)$ is a frame
\cite{Jans}.

\end{remark}



We close with a simple, but interesting example.
\begin{example}\label{example1} \rm
Let $W$ be a fundamental domain for $  1/2 \Z\times \Z$, $(x,y)\in \R^2$,
and $\Omega=W+(x,y)\cup W$. Then $(\chi_\Omega, \   1/2 \Z\times \Z
\times \Z \times \Z)$ is a frame for $L^2(\R^2)$ if $\Omega$ is a packing set
for $(\Z\times\Z)^{-T}=\Z\times\Z$. A simple computation shows that this is
the case if and only if $W\cap W \pm (k+x,\ell+y)=\varnothing$ for all
$(k,\ell)\in \Z^2$. If $\Omega$ is a packing set  for $\Z\times\Z$, then
$(\chi_\Omega, \ \tfrac 1 2 \Z\times \Z \times \Z \times \Z)$ is a tight frame
for $L^2(\R^2)$ if and only if $x\in(2\Z+1)$ and $y\in \Z $.  For a detailed
discussion see \cite{PRtr}
\end{example}

\section{Existence of smooth and compactly supported Gabor frame windows}\label{section:smoothwind}

Han and Wang's construction of Gabor orthonormal bases for separable
lattices is a landmark result in Gabor analysis.  The drawback of their
approach  lies in the fact that the constructed functions are discontinuous
and may not even decay at infinity. In this section, we will construct
nonnegative, compactly supported, and smooth window functions for a class
of separable lattices.

We begin this section with a simple result to indicates a limitation of
Theorem~\ref{th:ONBexist} as well as the direction which we will take. Its
derivation is included in the appendix.

\begin{proposition}\label{prop:noexistance-ronshen} Let $\Lambda=A\Z^d\times B\Z^d$ with $d(\Lambda)>1$.
Let $\Omega$ be a fundamental domain for $A\Z^d$ and a packing set  for
$B^{-T}\Z^d$. If $g\in C(\R^d)$ and $\supp g=\Omega$, then the Gabor
system $(g,\Lambda)$ is complete, but not a frame for $L^2(\R^d)$.
\end{proposition}

In view of Proposition~\ref{prop:noexistance-ronshen} we have to consider
window functions in $C_c^\infty(\R^d)$ whose support extends beyond the
fundamental domain $\Omega$ of $A\Z^d$. Recall that
$\Omega_\epsilon=\{x+y,x\in \Omega,\|y\|_2\le \epsilon\}$. Further, $\phi\in
C_c^\infty(\R^d)$ is nonnegative and  satisfies $\supp \phi \subseteq
B_1(0)=\{x\in \R^d: \|x\|_2\leq 1\}$ and $\int \phi =1$. Further $\phi_\epsilon
(x) = 1/\epsilon \, \phi(x/\epsilon)$.

\begin{theorem}\label{th:main} \

  \begin{enumerate}

         \item If there exists $\Omega\subseteq \R^d$, $\epsilon>0$, such
             that $\Omega$ is a fundamental domain for $A\Z^d$ and
             $\Omega_\epsilon$ is a packing set  for $B^{-T}\Z^d$, then
             $\displaystyle (\sqrt{\chi_\Omega \ast \phi_\epsilon},A\Z^d\times
             B\Z^d)$ is a tight frame for $L^2(\R^d)$ with frame bound $1/|\det
             B|$.

              \item  Suppose there exists $\Omega\subseteq \R^d$,
                  $\epsilon>0$, such that $\Omega$ is a fundamental domain for
        $B^{-T}\Z^d$ and $\Omega_\epsilon$ is a packing set  for $A\Z^d$,
        then $\displaystyle (\sqrt{\chi_\Omega \ast \phi_\epsilon},A\Z^d\times
        B\Z^d)$ is an orthonormal system.
  \end{enumerate}

\end{theorem}

Note that the conditions on $\Omega$ imply $d(\Lambda)<1$ in the first
statement and $d(\Lambda)>1$ in the second statement of
Theorem~\ref{th:main}.

\begin{proof}
We shall prove the second statement, the first statement follows then from
Theorem~\ref{th:ronshen} and the observation that
$$\|\sqrt{\chi_\Omega \ast
\phi_\epsilon}\|^2_{L^2}=\int |\chi_\Omega\ast \phi_\epsilon|=\int \chi_\Omega=m(\Omega)=|\det A|.$$

Following the proof of Theorem~\ref{th:ONBexist}, we obtain that
$(\chi_\Omega , \Lambda)$ is an orthonormal system. Moreover, $\supp
\chi_\Omega\ast \phi_\epsilon \subseteq \Omega_\epsilon$, and as
$\Omega_\epsilon$ is a packing set  for $A\Z^d$, we maintain
$\pi(k,\ell)\sqrt{\chi_\Omega\ast \phi_\epsilon}$ is orthogonal to $
\pi(k',\ell')\chi_\Omega\ast \phi_\epsilon$ if $k\neq k'$. It remains to show
that  $\{\pi(0,\ell')\chi_\Omega \ast \phi_\epsilon\}$ is orthogonal.  But this
follows as for $\ell\neq \ell'$
\begin{eqnarray}
 \langle \pi(0,\ell)\sqrt{\chi_\Omega \ast \phi_\epsilon} &\!\!\!\!\!\!\!\!\!, & \!\!\!\!\!\!\!\!\!  \pi(0,\ell')\sqrt{\chi_\Omega \ast \phi_\epsilon}\rangle
    =  \int e^{2\pi i (\ell-\ell') x }\chi_\Omega \ast \phi_\epsilon(x) \, dx \notag \\
     &=&  (\chi_\Omega \ast \phi_\epsilon)\widehat{\ \ } (\ell'-\ell)=  \widehat{\chi_\Omega}(\ell'-\ell)
     \widehat{ \phi_\epsilon} (\ell'-\ell)=0
     \,.\notag
\end{eqnarray}

\end{proof}

To construct sets that allow for the application of Theorem~\ref{th:main}, we
turn to lattices which have  starshaped common fundamental domains.%

\begin{proposition}\label{prop:starshaped}
If $0< |\det A|<|\det B|$ and $|\det B / \det A|^{1/d} A\Z^d$ and $B\Z^d$
have a bounded  and star-shaped common fundamental domain, then exists
$\Omega\subseteq \R^d$ such that $\Omega_\epsilon$ is a packing set  for
$B\Z^d$ and $\Omega$ is a tiling set for $A\Z^d$.
\end{proposition}

\begin{proof}
We let $\tilde A= |\det B / \det A|^{1/d} A\Z^d$, and obtain $|\det \tilde
A|=|\det B|$. By Theorem~\ref{th:hanwang} there exists a measurable set
$\Omega'$ which is a common fundamental domain for $\tilde A\Z^d$ and
$B\Z^d$, and by hypothesis, we can assume $\Omega'$ is star-shaped. We
claim that there exists a fundamental domain $\Omega$ for $A\Z^d$ such
that $\Omega_\epsilon\subset\Omega'$. By our hypothesis, we choose a
point $N\in\Omega'$ such that for all $P\in\Omega'$, the segment
$\overrightarrow{NP}$ is contained in the interior of $\Omega'$. Without loss
of generality we may take $N$ to be the origin
(Figure~\ref{figure:dilation-omega}). We apply a dilation with center
\emph{N} and coefficient $|\det B / \det A|^{-1/d}$ to $\Omega'$ and obtain
a set $\Omega$ with
$\Omega\cap\Omega'=\Omega$. 
Clearly, $\Omega$ is a fundamental domain for the lattice $A\Z^d$.
\end{proof}

\begin{corollary}\label{cor:main}
For $\Lambda=A\Z^d\times B\Z^d$ with $d(\Lambda)>1$ suppose that the
lattices $d(\Lambda)^{1/d}A\Z^d$ and $B^{-T}\Z^d$ have a bounded and
star-shaped common fundamental domain $\Omega$, then exists a
nonnegative $g\in \C_c^\infty(\R^d)$ with $(g,\Lambda)$ being a tight frame
for $L^2(\R^d)$.
\end{corollary}
\begin{proof}
 The result follows from Theorem~\ref{th:main}, part 2, and Proposition~\ref{prop:starshaped} where we replace $B$
 by $B^{-T}$ and note that
 $$1<d(\Lambda)=\Big|\frac 1 {\det A\, \det B} \Big|=\Big|\frac {\det B^{-T}} {\det A}  \Big|$$
\end{proof}
\begin{figure}
\begin{center}
   \includegraphics[width=7cm]{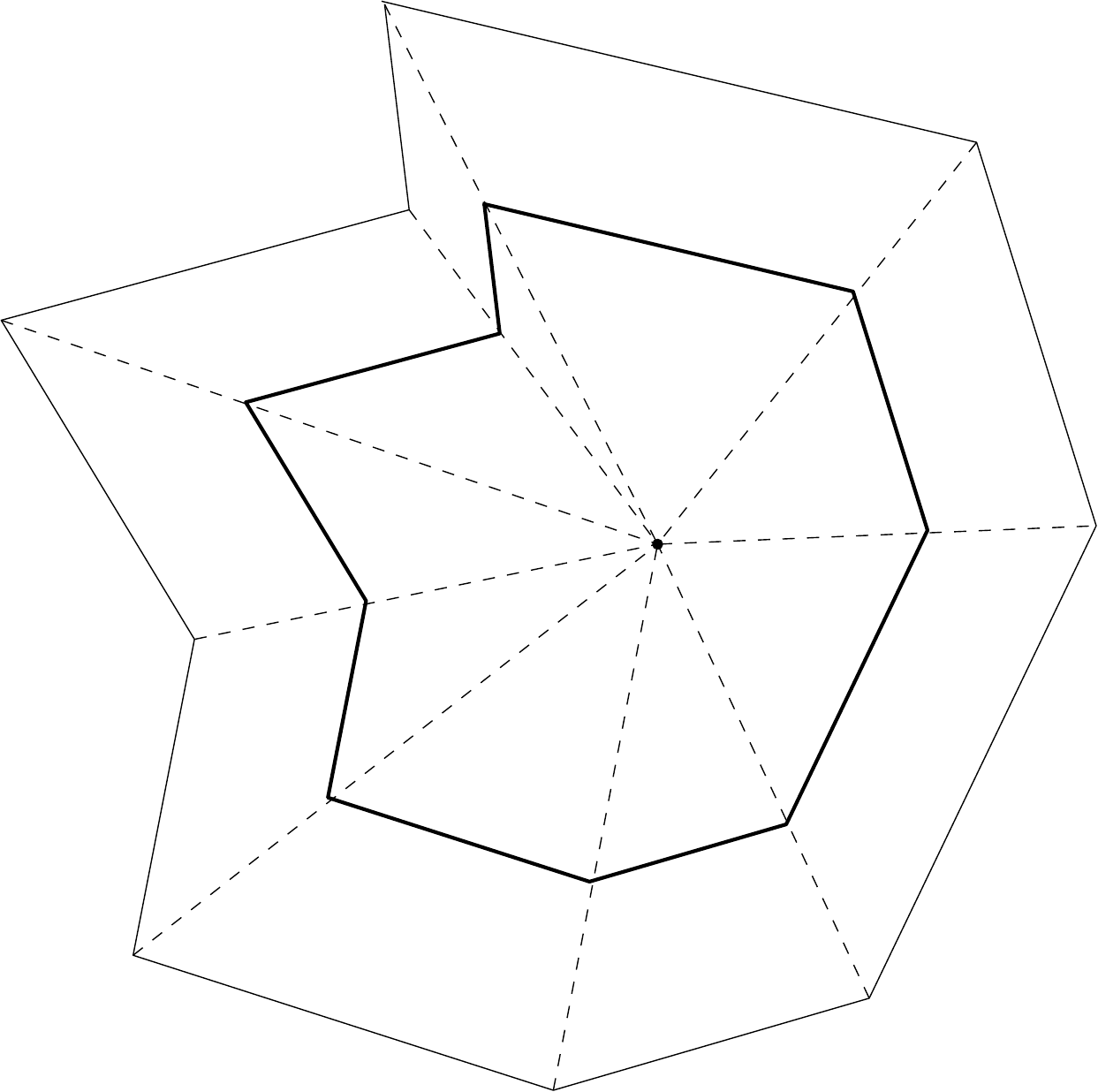}\\
   \end{center}
   \vspace{-4.8cm}\hspace{7.2cm}\emph{N}

   \vspace{1.6cm}\hspace{8.2cm}$\Omega$
   \hspace{1.5cm}$\Omega'$

\vspace{1.4cm}
   \begin{caption}{$\Omega$ is the scaled copy of $\Omega'$
   under dilation with center $N$.}\label{figure:dilation-omega}
\end{caption}
\end{figure}

Corollary~\ref{cor:main} can be extended to a class of upper and
lower-block triangular lattices.
\begin{corollary}\label{co:ltm}
If the lattices $ A\Z^d$ and $|\det A\, \det B|^{1/d}\,B^{-T}\Z^d$ have a
bounded and star-shaped common fundamental domain $\Omega'$ and if
$D$ is such that $DA^{-1}$  symmetric, then exists $g\in C_c^\infty(\R^d)$
such that $(g,\left(\begin{smallmatrix}
       A & 0 \\
       D & B \\
      \end{smallmatrix}\right)
    \Z^{2d})$ is a frame for $\LtR$.
\end{corollary}

\begin{proof}
Let
$T= \left(\begin{smallmatrix}
       I & 0 \\
       -DA^{-1} & I \\
\end{smallmatrix}\right)
$.
Then $T\left(\begin{smallmatrix}
       A & 0 \\
       D & B \\
      \end{smallmatrix}\right)\Z^d=A\Z^d\times B\Z^d$ is separable and fulfills the
conditions of Corollary~\ref{cor:main}, so there exists $\tilde g\in
C_c^\infty(\R^d)$ such that $(\tilde g,T\Lambda)$ is a frame for $\LtR$.
Since $T$ is symplectic, by Theorem~\ref{th:symplectic lattices} there
exists $g\in \calS(\R^d)$ such that $(g,\Lambda)$ is a frame for $\LtR$.
Furthermore, $g\in C_c^\infty(\R^d)$ because the metaplectic operator
associated to $T$ is a multiplication by a chirp, which preserves the
compact support of $\tilde g$ \cite{Fol89}.
\end{proof}

\section{Bivariate examples}\label{section:examples} In this section we provide several examples
which illustrate the geometric criteria established above for a family of
matrices in the case $d=2$.
\begin{proposition}\label{cor:common3matrices}
Let $q\in\Q^+$ and $m,n$ co-prime integers such that $q=m/n$. There
exists a common {convex} fundamental domain for $\Z^2$,
$\left(\begin{smallmatrix}
  q & 0 \\
  0 & 1/q \\
\end{smallmatrix}\right)\Z^2$ and
$\left(\begin{smallmatrix}
   1/ n & m \\
  0 & n \\
  \end{smallmatrix}\right)\Z^2$. Similarly, the lattices $\Z^2$, $\left(\begin{smallmatrix}
  n/m  & 0 \\
  0 & m/n \\
  \end{smallmatrix}\right)\Z^2$
and $\left(\begin{smallmatrix}
  n & 0 \\
 m & 1/n  \\
  \end{smallmatrix}\right)\Z^2$,
have a common convex fundamental domain.
\end{proposition}
\begin{proof}
We shall only prove the first assertion, the second one follows analogously.

Let $\Omega=\left(\begin{smallmatrix}
1/n  & m \\
0 & n \\
\end{smallmatrix}\right)[0,1)^2$.
$\Omega$ is a convex set and $m(\Omega)=1$. Suppose that there exists
$(k,l)^T\neq\overrightarrow{0} \in \Z^2$ such that $\{\Omega +
(k,l)^T\}\cap\Omega\neq\varnothing$. Then there exist points
$(\alpha_1,\beta_1), (\alpha_2,\beta_2)\in[0,1)^2$ such that
$\tfrac{\alpha_1}n+m\beta_1+k=\tfrac{\alpha_2}n+m\beta_2
\quad\text{and}\quad n\beta_1+l=n\beta_2$. Therefore,
$\beta_2-\beta_1=\tfrac ln$, which implies that
$\alpha_1-\alpha_2=ml-kn\in\Z$. Since $0\le\alpha_1,\alpha_2<1$,
necessarily $\alpha_1=\alpha_2$, and also that $\beta_2-\beta_1=\tfrac
km$. Since $\gcd(m,n)=1$ and $0\le \beta_2-\beta_1<1$, this is possible
only if $k=l=0$. Thus
$(\Omega+\Z^2{\setminus}\{0\})\cap\Omega=\varnothing$. Since
$m(\Omega)=1$, $\Omega$ is a fundamental domain for $\Z^2$.

A similar proof shows that $\Omega$ is also a fundamental domain for the lattice $\left(\begin{smallmatrix}
  m/n & 0 \\
  0 & n/m  \\
\end{smallmatrix}\right)\Z^2$.
\end{proof}

\begin{figure}
\[
 \beginpicture \color{black}
 \setcoordinatesystem units <1cm,1cm> \linethickness 0.3mm
 %
 %
 %
 %
 \arrow <0.4cm> [0.375,0.75]     from 1 0.5 to 1 10
 \arrow <0.4cm> [0.375,0.75]     from 0.5 1 to 7 1
 %
     \setsolid
     \plot 4.5 0.9 4.5 1.1 /
      \plot 2 0.9 2 1.1 /
     \put {\small 0}     at 0.7 0.6
     \plot 5.5 0.9 5.5 1.1 /
     \put {$m+\tfrac1n$}     at 5.5 0.4
     \put {$x_1$} at 7.5 1
     \put {$x_2$} at 0.7 9.5
\plot 0.9 9 1.1 9 /
 %
         \linethickness 2mm
           \plot 1 1 4.5 9 /
\plot 2 1 5.5 9 /
\plot 4.5 9 5.5 9 /
\put {\textit{\small m}} at 4.5 0.4
\put {\textit{\small n}} at 0.7 9
\put {$\tfrac 1n$} at 2 0.4
\put {$\Omega$} at 3.2 5
 \endpicture
 \]
 \begin{caption}{The set $\Omega$ constructed in Proposition~\ref{cor:common3matrices}.}\label{figure:example}
\end{caption}
\end{figure}

To illustrate strength and weakness of our method, we shall consider the
following, apparently simple example.
\begin{corollary}\label{prop:diagonal}
If $\Lambda=a\Z\times b\Z\times c\Z\times d\Z$  with
$d(\Lambda)=abcd<1$ satisfies
\begin{eqnarray}
  (1)\ ac<1,bd<1,\quad \text{or}\quad (2)\  abcd<1/2 ,\quad \text{or}\quad (3)\  \sqrt{\frac{ac}{bd}}\in \Q\notag
  \end{eqnarray}
 then there exists $g\in C_c^\infty(\R^2)$ such that $(g,\Lambda)$
is a tight frame for $L^2(\R^2)$.
\end{corollary}
\begin{proof} (1) If $ac<1$ and $bd<1$, then any $g_1\in
C_c^\infty(\R)$ with $\chi_{[0,a]}\leq g_1\leq \chi_{[(ac-1)/(2c),(ac+1)/(2c)]}$
and
  $g_2\in
C_c^\infty(\R)$ with $\chi_{[0,b]}\leq g_2\leq \chi_{[(bd-1)/(2d),(bd+1)/(2d)]}$
guarantees that $ (g_1,a\Z\times c\Z)$ and $(g_2,b\Z\times d\Z)$ are
frames. A simple tensor argument then implies that $(g_1\otimes
g_2,a\Z\times b\Z\times c\Z\times d\Z)$ is a frame for $L^2(\R^d)$. Note that
the same line of argument shows that if either $ac>1$ or $bd>1$, then
exists no $g_1,g_2\in C_c^\infty(\R)$ with $(g_1\otimes g_2,a\Z\times
b\Z\times c\Z\times d\Z)$ is a frame for $L^2(\R^d)$ \cite{PRtr}.

\noindent (2) It suffices to consider $abcd<1/2$, and $ac>1$ or $bd>1$, as
else, (1) would apply. Without loss of generality, we consider $ac>1$, and,
hence $bd<1/2$. Moreover, applying Theorem~\ref{th:symplectic lattices}
with the symplectic matrix $M=\diag(c, d, 1/c, 1/d)$ implies  that the
existence of $g\in C_c^\infty(\R^2)$ with $(g,\Lambda)$ being a frame for
$L^2(\R^2)$ follows from the respective statement for $\Lambda'=\diag(c, d,
1/c, 1/d)\Lambda=\diag(ac, bd, 1, 1)\Z^4$. With
$\Omega=\left(\begin{smallmatrix}
    ac  & 0 \\
    1/2  & bd  \\
\end{smallmatrix}\right)[0,1)^2$, $\epsilon =\frac{1-2abcd}{4ac}$, any $g\in C_c^\infty(\R^2)$ with
$$ \chi_\Omega \leq g \leq \chi_{\Omega_\epsilon} $$ has the property that
$(g,\Lambda')$ is a frame for $L^2(\R^2)$.

\noindent (3) Theorem~\ref{th:separable1} applies whenever there exist
$m,n\in\Z,\alpha\in\R$ such that $\alpha m^2=ac$ and $\alpha n^2=bd$,
which is equivalent to $\sqrt{\tfrac{ac}{bd}}\in\Q$.
\end{proof}

The conditions on $\Lambda$ presented in Proposition~\ref{prop:diagonal}
are not necessary for the existence of $g\in C_c^\infty(\R^2)$ with
$(g,\Lambda)$ being a Gabor frame for $L^2(\R^d)$.  In fact, if for $g\in
C_c^\infty(\R^2)$, $(g,M\Z^{4})$ is a Gabor frame for $L^2(\R^2)$, then
exists an open neighborhood $U$ of $M$ in $GL(\R^4)$, such that
$(g,M'\Z^{4})$ is a Gabor frame for $L^2(\R^d)$ for all $M'\in U$
\cite{FeiKai}. For results where rationality of lattices plays a central role, see
results known as Janssen's tie \cite{Jans}.

Below, we show that the condition $abcd<1$ and the use of diagonal
matrices with rational entries  in Corollary~\ref{prop:diagonal} is critical for
our method to be applicable.    This clearly  illustrates the limitations of the
method described in Theorem~\ref{th:ONBexist}.

\begin{proposition}\label{prop:example-no-cssfd}
There exists no fundamental domain $\Omega$ for $\Z\times \frac 1 2 \Z$
with $\Omega_\epsilon$, $\epsilon>0$, is a packing set  for $\Z \times \Z$.
Consequently, there exists no common star-shaped fundamental domain for
$\Z^2$ and $\left(\begin{smallmatrix}
    \sqrt2  & 0 \\
    0  & \frac1{\sqrt2}  \\
\end{smallmatrix}\right)\Z^2$.
\end{proposition}
\begin{proof}
Suppose that
 $\Omega$ is a tiling set for $\Z\times \frac 1 2 \Z$. If  $\Omega_\epsilon$ is a packing set  for
 $\Z\times \Z$, then
$\{\Omega+(m,n):m,n\in\Z\}$ have no boundary points in common.
Hence all sets $\{\Omega+(m,\tfrac n 2):m,n\in\Z\}$ with common boundary
point with $\Omega$ must be of the form $\Omega+(m,n+\tfrac12)$.
Clearly, there must be two such sets in the corona of $\Omega$ which have
a common boundary points. But this is a contradiction as the system
$\{\Omega+(m,n+\tfrac12):m,n\in\Z\}$ is a translate of the system
$\{\Omega+(m,n):m,n\in\Z\}$ and hence all of its members should have
disjoint boundaries.

To obtain the second assertion, assume that there exists a compact
star-shaped set $\Omega'$ serving as a common fundamental domain for
both lattices. Then there exists $x\in\R^2,\epsilon>0$ such that for
$\Omega=D_{\frac1{\sqrt2}}\Omega'+x$, we have
$\Omega\subseteq\Omega_\epsilon\subseteq\Omega'$. Note that
$\Omega$ is a tiling set
for $\frac1{\sqrt2}\left(\begin{smallmatrix} \sqrt2&0\\
0&\frac{\sqrt2}2
\end{smallmatrix}\right)\Z^2
=\left(\begin{smallmatrix} 1&0\\
0&\frac{1}2
\end{smallmatrix}\right)\Z^2$, contradicting the first assertion.

\end{proof}

\begin{theorem}\label{th:separable1}
Let $m,n\in\Z$ be relatively prime. Let $\Lambda=A\Z^2\times B\Z^2$ be a lattice in $\R^4$. Whenever $B^TA$ is of the form
\begin{enumerate}
          \item $\alpha I$, $|\alpha |< 1$;
          \item  $\left(\begin{smallmatrix}
    m^2\alpha  & 0 \\
    0 & n^2\alpha  \\
  \end{smallmatrix}\right)
$, where $|\alpha |< (mn)^{-1}$;
\item $
\left(\begin{smallmatrix}
    \alpha  & mn\alpha  \\
    0 & n^2\alpha  \\
\end{smallmatrix}\right)$, where 
  $|\alpha |<n^{-1}$;
 or
          \item $
\left(\begin{smallmatrix}
    n^2\alpha  & 0 \\
    mn\alpha  & \alpha  \\
\end{smallmatrix}\right)$, where 
$|\alpha |<n^{-1}$,
        \end{enumerate}
then there exists a function $g\in C_c^\infty(\R^2)$ such that
$(g,\Lambda)$ is a Gabor frame for $L^2(\R^2)$.
\end{theorem}
\begin{proof}
We have
\[    \begin{pmatrix}
      A & 0 \\
      0 & B \\
    \end{pmatrix}=
    \underbrace{\begin{pmatrix}
      B^{-T} & 0 \\
      0 & B \\
    \end{pmatrix}}_M
    \begin{pmatrix}
      B^TA & 0 \\
      0 & I \\
    \end{pmatrix}.
\]
This shows that $\Lambda=M((B^TA)\Z^2\times\Z^2)$ with $M$ symplectic.
Since $|\det B^TA| \le1$, we can rescale $B^TA\Z^2$ to make its density 1.
Then $ |\det B^T A|^{-1/2} B^TA \Z^2$ is respectively of the form
\[\Z^2, \begin{pmatrix}
  m/n & 0 \\
  0 & n/m  \\
\end{pmatrix}\Z^2, \begin{pmatrix}
  1/n & m \\
  0 & n \\
\end{pmatrix}\Z^2,\begin{pmatrix}
  n & 0 \\
  m & 1/n \\
\end{pmatrix}\Z^2\] Proposition~\ref{cor:common3matrices} assures
the existence of a common convex fundamental domain for $(B^TA)\Z^2$
and $\Z^2$ accordingly. By Theorem~\ref{th:ONBexist} there exists $g'\in
C_c^\infty(\R^2)$ such that $(g',(B^TA)\Z^2\times\Z^2)$ is a Gabor frame for
$L^2(\R^2)$. The matrix $M$ is symplectic, and its associated metaplectic
operator $\mu(M)$ from Theorem~\ref{th:symplectic lattices} is the dilation
$(\mu(M)h)(x)=|\det B|^{-\frac12}h(B^{-1}x)$, \cite{Fol89}. Hence,
$g=\mu(M)^\ast g'\in C_c^\infty(\R^2)$ and $(g,\Lambda)$ is a Gabor frame
for $L^2(\R^2)$.
\end{proof}

Note that it is not known for which $\alpha,\beta$ there exists
$g\in\calS(\R^2)$ such that $(g,\Z^2\times\left(\begin{smallmatrix} \alpha&0\\
0&\beta
\end{smallmatrix}\right)$ is Gabor frame for $L^2(\R^d)$. If $g_0(x)=e^{-\pi\|x\|_2^2}$ is a
Gaussian, then $(g_0,\Z^2\times\left(\begin{smallmatrix} \alpha&0\\
0&\beta
\end{smallmatrix}\right)$ is a frame for $L^2(\R^d)$ if $\alpha,\beta<1$ \cite{PRtr}.

\section{Appendix: Proof of Proposition~\ref{prop:noexistance-ronshen}}

The operator \[S_{g,\Lambda}f=\sum_{\lambda\in\Lambda} \langle
f,\pi(\lambda)g\rangle\,\pi(\lambda)g,\quad f\in \LtR\] is called a Gabor frame
operator. It is positive and self-adjoint if $(g,\Lambda)$ is a frame for
$L^2(\R^d)$ \cite{G01,C03}. Gabor frames possess a very useful
reconstruction formula:
\begin{equation*}\label{eq:framexp}
f=\sum_{\lambda\in\Lambda}\langle f,\pi(\lambda)g\rangle\,
\pi(\lambda)\gamma=\sum_{\lambda\in\Lambda}\langle
f,\pi(\lambda)\gamma\rangle\,\pi(\lambda)g,\quad f\in\LtR
\end{equation*}
with $\gamma=S_{g,\Lambda}^{-1}g$ being the so-called canonical dual
window \cite{G01,C03}.

$S_{g,A\Z^d\times B\Z^d}$ can be represented in matrix form
\cite{Wa92,RS97}. For that purpose, we define the bi-infinite
cross-ambiguity Gramian matrix
\begin{equation}\label{eq:cross-ambig}
\begin{aligned}
&\mathbf{G}(x)=(G_{jk}(x))_{j,k\in\Z^d}:\\
&G_{jk}(x)=|\det B|^{-1}\sum_{\ell \in\Z^d}\overline{g(x-B^{-T}
k-A\ell )}g(x-B^{-T} j-A\ell ).
\end{aligned}
\end{equation}

 Below, $W(\R^d)$ denotes the Wiener space,
consisting of all functions such that the norm
\[\|f\|_W=\sum_{k\in\Z^d}\|f\cdot T_k\chi_{[0,1)^d}\|_\infty\] is finite \cite{C03}.

\begin{proposition}\label{prop:ron-shen-matrix} Let $g\in W(\R^d)$. Let $\Lambda=A\Z^d\times B\Z^d$ be a
full-rank lattice in $\R^{2d}$. For $f,h\in L^2(\R^d)$, define the sequences
\[\mathbf{f}(x):=\{f(x-B^{-T} j):j\in\Z^d\},\quad\mathbf{h}(x):=\{h(x-B^{-T}
k):k\in\Z^d\}.\] Then the following holds:
\begin{equation*}\label{eq:walnut-matrix}
\langle S_{g,\Lambda}f,h\rangle=\int_{B^{-T} [0,1)^d}\langle
\mathbf{G}(x)\mathbf{f}(x),\mathbf{h}(x)\rangle dx
\end{equation*}
for all $f,h\in L^2(\R^d)$. $S_{g,\Lambda}$ is a bounded operator on
$L^2(\R^d)$ if and only if there exists $b>0$ such that $\mathbf{G}(x)\le
bI_{\ell^2}$ for almost all $x\in \R^d$. Also $S_{g,\Lambda}$ is a boundedly
invertible operator on $L^2(\R^d)$ if and only if there exists $a>0$ such that
$\mathbf{G}(x)\ge aI_{\ell^2}$ for almost all $x\in \R^d$.
\end{proposition}
Note that $(g,\Lambda)$ is a Gabor frame for $\LtR$ if and only if
$S_{g,\Lambda}$ is bounded and boundedly invertible on $\LtR$.

 {\it Proof of  Proposition~\ref{prop:noexistance-ronshen}.}
Let $\Omega$ be a fundamental domain for $A\Z^d$ and packing set  for
$B^{-T}\Z^d$, $g\in C(\R^d)$ with $\supp g=\Omega$. Let $f\in \LtR$.
Denote by $f_k$ the restriction of $f$ to
$\Omega+Ak$, $k\in\Z^d$. Thus $\|f\|_2^2=\sum_{k\in\Z^d}\|f_k\|_2^2$, where $f_k\in L^2(\Omega+Ak)$.

To prove completeness, suppose there exists $f\in \LtR$ such that
\begin{equation}\label{eq:ftsuppg}
\langle f, M_{Bl}T_{Ak}g\rangle =0,\quad  k,l\in\Z^d.\end{equation} However,
because $\supp g=\Omega$, for a fixed $k\in\Z^d$,~\eqref{eq:ftsuppg} is
the Fourier coefficient $(f_k\cdot T_{Ak}g)\widehat{\ }(Bl)$.
From the Fourier series expansion \[(f_k\cdot T_{Ak}g)(t)=\sum_{l\in\Z^d}
(f_k\cdot T_{Ak}g)\widehat{\ }(Bl)\,e^{2\pi iBl\cdot t}\] we see that $f_k\cdot
T_{Ak}g$ is identically 0 almost everywhere on $\Omega+Ak$. Because
$g$ does not vanish on a subset of $\Omega$ of positive measure,
$f_k=0$ almost everywhere on $\Omega+Ak$ for all $k$. Therefore $f=0$
almost everywhere on $\R^d$ and completeness of $(g,\Lambda)$ is
shown.

Assume now that $(g,A\Z^d\times B\Z^d)$ is a Gabor frame for $\LtR$. We
analyze the structure of the associated cross-ambiguity matrix
$\mathbf{G}(x)$. If $j\neq k$, then
\[\supp g(x-B^{-T} k-Al)g(x-B^{-T} j -Al)\subseteq
Al+[(\Omega+B^{-T} k)\cap(\Omega+B^{-T} j)].\] Since
$m((\Omega+B^{-T} k)\cap(\Omega+B^{-T} j))=0$, for almost all $x$ we
have $G_{kj}(x)=0$ for $j\neq k$. Thus the matrix $\mathbf{G}(x)$ given
by~\eqref{eq:cross-ambig} is diagonal for almost all $x$. Moreover,
\begin{equation*}\label{eq:me}
G_{00}(x)=\sum_{l\in\Z^d}|g(x-Al)|^2=|g(x)|^2,\end{equation*} since $\supp
g=\Omega$. When $\mathbf{c}=\{\delta_{0}(n)\}_{n\in\Z^d}$,
Proposition~\ref{prop:ron-shen-matrix} implies that $a\le\langle \mathbf{G}(x)
\mathbf{c},\mathbf{c}\rangle\le b$, because $S_{g}$ is a bounded and
invertible operator on $L^2(\R^d)$. Therefore, the associated matrix
$\mathbf{G}(x)$ shares these properties for almost every $x$. But then
$\langle\mathbf{G}(x) \mathbf{c},\mathbf{c}\rangle=G_{00}(x)=|g(x)|^2$,
which in turn implies $a\le |g(x)|^2\le b$ on $\Omega$, contradicting the
continuity of $g$ on $\R^d$.\hfill $\square$

%
\bibliographystyle{alpha}
\bibliography{bibliogr, ../Bibliography/gabor_goetz}

\end{document}